\documentclass{amsart}
\usepackage[utf8]{inputenc}
\usepackage{amssymb}
\usepackage{amsmath}
\usepackage{amsthm}
\usepackage{orcidlink}
\usepackage{hyperref}
\usepackage{xcolor}
\usepackage{enumitem}
\numberwithin{equation}{section}
\usepackage[left=4cm,top=4cm,right=4cm,bottom=4cm]{geometry}
\usepackage[initials]{amsrefs}

\newcommand{\IR}{\mathbb{R}}
\newcommand{\IN}{\mathbb{N}}

\newcommand{\Lip}{\operatorname{Lip}}
\newcommand{\lip}{\operatorname{lip}}
\newcommand{\ILip}{\operatorname{\mathbb{L}ip}}
\newcommand{\Int}{\operatorname{Int}}

\theoremstyle{theorem}
\newtheorem{theorem}{Theorem}[section]

\newtheorem{proposition}[theorem]{Proposition}

\theoremstyle{definition}

\newtheorem{problem}[theorem]{Problem}

\begin{document}
\title[Sets in metric spaces where the big Lipschitz derivative is infinite]{Characterization of sets in metric spaces where the big Lipschitz derivative is infinite}

\author[Llorente]{Jes\'us Llorente*  \orcidlink{0000-0001-7528-7366}}
\address[Jes\'us Llorente]{\mbox{}\newline\indent Instituto de Matem\'atica Interdisciplinar (IMI), Departamento de An\'{a}lisis Matem\'{a}tico y Matem\'atica Aplicada, Facultad de Ciencias Matem\'aticas, Plaza de Ciencias 3, Universidad Complutense de Madrid, 28040 Madrid, Spain.}\email{jesllore@ucm.es}

\author[Maslyuchenko]{Oleksandr Maslyuchenko \orcidlink{0000-0002-1493-9399}} 
\address[Oleksandr Maslyuchenko]{\mbox{}\newline\indent Institute of Mathematics, Bankowa 12, University of Silesia in Katowice, 40-007 Katowice, Poland. }\email{ovmasl@gmail.com}

\author[Wójcicki]{Ziemowit Wójcicki \orcidlink{0009-0005-6433-0708}} 
\address[Ziemowit Wójcicki]{\mbox{}\newline\indent Institute of Mathematics, Bankowa 12, University of Silesia in Katowice, 40-007 Katowice, Poland.}\email{ziemo1@onet.eu}


\thanks{The research was supported by the University of Silesia in Katowice, Mathematics Department (Iterative Functional Equations
and Real Analysis program). The first author was also supported by Grant PID2025-168101NB-I00 funded by MICIU/AEI/10.13039/501100011033, and by Grant PID2022-138758NB-I00 funded by the Spanish Ministry of Science and Innovation.}
\thanks{*Corresponding author}

\begin{abstract}
Let $X$ be a metric space without isolated points. In this paper we prove that $E$ is a $G_\delta$-set in $X$ if and only if there exists a continuous function $f\colon X\to\IR$ whose big Lipschitz derivative
is infinite precisely on $E$. To this end, we introduce a new generalization of the classical Takagi function to the setting of metric spaces. 
\end{abstract}

\maketitle

\section{Introduction}

In the one dimensional case, the classical Rademacher theorem states that every Lipschitz function $f:\mathbb{R}\to \mathbb{R}$ is differentiable almost everywhere. The Lipschitz assumption in this theorem can be replaced by a weaker condition formulated in terms of the big Lipschitz derivative of the function, while preserving the same conclusion.

For a continuous function $f:\mathbb{R}\to \mathbb{R}$, the big Lipschitz derivative is defined by 
\begin{equation}\label{LipDef}
\Lip f(x)=\limsup_{r\to 0^+}\,\sup_{u\in (x-r,x+r)}\frac{|f(u)-f(x)|}{r}.
\end{equation}
The Rademacher-Stepanov theorem asserts that $f$ is differentiable almost everywhere on the set $\{x\in \mathbb{R}: \Lip f(x)<\infty\}$. This result generalizes Rademacher's theorem since $f$ is $K$-Lipschitz if and only if $\Lip f(x)\leq K$ for all $x\in \mathbb{R}$. 

By replacing the lim sup in definition \eqref{LipDef} with a lim inf, we obtain the little Lipschitz derivative, which is denoted by $\lip f$. This notion appeared in the work of J.~Cheeger (see \cite{Cheeger}) and has since been extensively
studied in the literature (see \cite{BuHaMaVe2020}, \cite{HLL}, \cite{HLLS}, \cite{Keith} or \cite{RmZu}, for instance). Note that if $f$ is differentiable at $x\in\mathbb{R}$, then
$$
\Lip f(x)=\lip f(x)=|f'(x)|.
$$
Furthermore, it has been shown that the Rademacher–Stepanov theorem fails when the set $\{x\in\mathbb{R}:\Lip f(x)<\infty\}$ is replaced by $\{x\in\mathbb{R}:\lip f(x)<\infty\}$ (see Theorem 1.4 of \cite{BaCs}).

The local Lipschitz derivative of a continuous function $f:\mathbb{R}\to \mathbb{R}$ is defined by 
\begin{equation}\label{LocalDef}
\ILip f(x)=\inf_{r>0}\, \,\sup_{\substack{u,v\in (x-r,x+r)\\ u\neq v}} \frac{|f(u)-f(v)|}{|u-v|}.
\end{equation}
This quantity was considered by V. Herasymchuk and the second author in the more general setting of metric spaces (see \cite{HerMas}). Note that if $f$ is differentiable, then 
$$
\ILip f(x) = \limsup_{u\to x}|f'(u)|.
$$
Moreover, if $f$ is of class $C^1$, then $\ILip f(x)=|f'(x)|$.

The Lipschitz derivatives of a function can behave very differently. For instance, there exists a continuous nowhere differentiable function  $f:\mathbb{R}\to \mathbb{R}$ such that $\Lip f(x)=\infty$ for all $x\in \mathbb{R}$ and $\lip f(x)=0$ except on a set of Hausdorff dimension zero (see Theorem~2.3 of \cite{H}). Interestingly, in the sense of Baire category, a typical continuous function $f:\mathbb{R}\to\mathbb{R}$ satisfies $\Lip f(x)=\infty$ for every $x\in\mathbb{R}$ and  $\lip f(x)=0$ for almost every $x\in\mathbb{R}$ (see \cite{Ba} and Theorem 4.2 in \cite{BuHaRmZu}).

On the other hand, there are certain restrictions on the
possible behavior of these quantities. In particular, there is no continuous function $f:\mathbb{R}\to\mathbb{R}$
such that $\Lip f(x)=\infty$ and $\lip f(x)<\infty$ for every $x\in\mathbb{R}$.  This is a consequence of the fact that if
$f:\mathbb{R}\to\mathbb{R}$ is continuous and the set $\{x\in\mathbb{R}:\lip f(x)=\infty\}$ is countable, then $f$ is differentiable on a set of positive Lebesgue measure (see Lemma 1.1 in \cite{BaCs}).

Following this line of research, considerable work has recently been devoted to characterizing those sets
$E\subseteq\mathbb{R}$ for which there exists a continuous function $f:\mathbb{R}\to\mathbb{R}$ such that its Lipschitz derivatives exhibit prescribed behaviors precisely on $E$ (see \cite{BuHaMaVe2021}, \cite{Ha} or \cite{Ha1}, for instance). In particular, it has been proved that a subset $E$ of $\mathbb{R}$ is a $G_\delta$ set if and only if there exists a continuous function $f:\mathbb{R}\to\mathbb{R}$ such that  $\{x\in \mathbb{R}: \Lip f(x)=\infty\}=E$ (see Lemma 2.4 and Theorem 3.35 of \cite{BuHaRmZu}).

Since Lipschitz derivatives can also be defined in metric spaces, it is natural to ask whether some of the previous results
can be extended to this more general setting. In this direction, the last two authors have studied the semicontinuity properties of Lipschitz derivatives of functions defined on metric spaces (see \cite{MaWo2026-1}).

Concerning the characterization of sets where the big Lipschitz derivative is infinite, they also proved that
if $E$ is an open subset of a metric space $X$ without
isolated points, then there exists a continuous function
$f:X\to\mathbb{R}$ such that the set 
$$
(\Lip f)^{-1}(\infty)=\{x\in X: \Lip f(x)=\infty\}
$$
equals $E$ (see \cite{MaWo2025}). Subsequently, a similar result was established for closed subsets of $X$ (see \cite{MaWo2026}). 

In this paper, we extend these results by proving that if $E$ is a $G_\delta$-set in a metric space $X$ without isolated points, then there exists a continuous function $f\colon X\to\IR$ such that $(\Lip f)^{-1}(\infty)=E$. To establish this result, we introduce a generalization of the classical Takagi function  to metric spaces.  Since the set $(\Lip f)^{-1}(\infty)$ is a $G_\delta$ set for every continuous function
$f\colon X\to\IR$ defined on a metric space $X$ (see Corollary 5.6 of \cite{MaWo2026-1}), our result provides a complete characterization of such sets in metric spaces without isolated points, extending the corresponding characterization
previously established for functions on the real line.

\section{Preliminaries}\label{sec:Preliminaries}
We begin by introducing the notation used throughout the paper. We write $ \IN_0 = \{0,1,2,\ldots\}$, $\IN= \{1,2,\ldots\}$  and use the symbol ``$:=$'' to denote equality by definition. 

For an arbitrary sequence of elements $a_n$, $n\in \IN_0$, we write $\boldsymbol{a}=(a_n)_{n=0}^\infty$. The boldface symbol $\boldsymbol{a}$ is reserved for this sequence, whereas the corresponding nonbold
symbol $a$ may be used for other objects. We follow the same convention for other sequences.

In all limit expressions, such as $\limsup_{u\to x}f(u)$ or $\limsup_{(u,v)\to (x,x)}f(u,v)$, the variables are understood to range over the domain of the function involved.

When taking infima and suprema of nonnegative quantities, we adopt the conventions $\inf\varnothing=\infty$ and $\sup\varnothing=0$.

Throughout the paper, $(X,d)$ is a metric space without isolated points.  For $x\in X$ and $r>0$, the open and closed balls centered at $x$ with radius $r$ are denoted by $B(x,r)$ and $B[x,r]$, respectively.

A function $f\colon X\to\IR$ is called Lipschitz if there is a constant $K>0$ such that $|f(x)-f(y)|\leq K d(x,y)$ for all $x,y\in X$. In such a case, it is said that $f$ is $K$-Lipschitz.

The distance function to a non-empty set $E\subseteq X$ is defined by
\begin{equation*}
	d_E(x)=d(x,E)=\inf_{y\in E}d(x,y) 
\end{equation*} 
for any $x\in X$. It is well known that $d_E$ is 1-Lipschitz.

The \emph{local} and \emph{big Lipschitz derivatives} of a function $f\colon X\to\IR$ are defined, respectively, by
\begin{align*}
	\Lip f(x)&=\limsup_{r\to 0^+}\sup_{u\in B(x,r)}\frac{\big|f(u)-f(x)\big|}{r}\, \,  \text{ and }\\
\ILip f(x)&=\inf_{r>0}\, \, \sup_{\substack{u,v\in B(x,r)\\ u\neq v}} \frac{|f(u)-f(v)|}{d(u,v)}
\end{align*}
for any $x\in X$. It is not hard to show that 
\begin{align*}
\Lip f(x)&=\limsup_{u\to x}\frac{\big|f(u)-f(x)\big|}{d(u,x)}\, \,\text{ and }\\
	\ILip f(x)&=\limsup_{(u,v)\to (x,x)}\frac{\big |f(u)-f(v)\big |}{d(u,v)}
\end{align*}
for any $x\in X$ (see Proposition 2.1 of \cite{MaWo2026-1} for the first equality). 

A function $f\colon X\to\IR$ is \emph{locally} (resp. \emph{pointwise}) \emph{Lipschitz} at a point $x\in X$ if $\ILip f(x)<\infty$ (resp. $\Lip f(x)<\infty$), and on a set $A\subseteq X$ if the corresponding condition holds at every point of $A$.

Note that all the previous definitions extend naturally to functions between two metric spaces.

\section{The Takagi function in the metric setting}

The Takagi function is a classical example of a continuous nowhere differentiable function (see \cite{Ta}).  It has attracted considerable interest among mathematicians over the years owing to its remarkable properties from various perspectives. In modern notation, it is defined by 
\begin{equation}\label{TakagiClassical}
\tau(x)=\sum_{n=0}^{\infty} d(x, \tfrac{1}{2^n}\mathbb{Z}), \quad x\in \mathbb{R}.
\end{equation}

Numerous generalizations of the Takagi function were introduced with the aim of extending its striking properties to broader families of functions. In the literature, these are referred to as \emph{generalized Takagi functions}. Most of them are defined on the real line, although generalized Takagi functions have recently been introduced on separable Hilbert spaces (see \cite{FGG}), and on metric spaces (see \cite{MaWo2025} and \cite{MaWo2026}).  Surveys \cite{AK} and \cite{L}, as well as the thesis \cite{LL}, provide extensive information on the Takagi function and its generalizations.  In this section, we introduce a new generalization of the Takagi defined on a metric space.

Let $\boldsymbol{\alpha}=(\alpha_n)_{n=0}^\infty$ and $\boldsymbol{\delta}=(\delta_n)_{n=0}^\infty$ be sequences of positive real numbers such that 
\begin{align}
	\label{equ:Takagi1}
	&\lim_{n\to\infty} \delta_n=0;\\
	\label{equ:Takagi2}
	& \sum_{n=0}^{\infty}\alpha_n\delta_n<\infty.
\end{align} 
Let $\boldsymbol{A}=(A_n)_{n=0}^\infty$ be a sequence of non-empty subsets of $X$ such that for every $n\in\IN_0$ the following conditions hold:
\begin{align}
	\label{equ:Takagi3}
	&A_n\text{ is a closed subset of }X;\\
        \label{equ:Takagi5}
       &A_n\subseteq A_{n+1}; \\
	\label{equ:Takagi4}
	&d(x,A_n)\le\delta_n\text{ for all }x\in X.
\end{align}  
Note that \eqref{equ:Takagi1} and \eqref{equ:Takagi4} imply 
\begin{equation*}
	\bigcup\boldsymbol{A}:=\bigcup_{n=0}^\infty A_n\, \text{ is a dense subset of }X.
\end{equation*}

The $\boldsymbol{\alpha}$-$\boldsymbol{A}$-\emph{Takagi function} $T=T_{\boldsymbol{\alpha},\boldsymbol{A}}:X\to[0,\infty)$  is defined by
\begin{equation*}
	T(x)=T_{\boldsymbol{\alpha},\boldsymbol{A}}(x)=\sum_{n=0}^{\infty}\alpha_n\varphi_n(x),
\end{equation*}
where $\varphi_n(x)=d(x,A_n)$. Note that, for the classical Takagi function defined in \eqref{TakagiClassical}, we have
$$
A_n = \tfrac{1}{2^n}\mathbb{Z}, \qquad
\delta_n = \tfrac{1}{2^{n+1}}, \qquad
\text{and} \qquad \alpha_n = 1\, \text{}
$$
for every $n$.

We also denote 
\begin{equation*}
	S_n(x)=\sum_{k=0}^{n-1}\alpha_k\varphi_k(x)
	\ \ \ \text{and}\ \ \ 
	R_n(x)=\sum_{k=n}^{\infty}\alpha_k\varphi_k(x).
\end{equation*}
for any $n\in\IN_0$ and $x\in X$. Thus, $T(x)=S_n(x)+R_n(x)$. Note that  for any $n\in \mathbb{N}_0$ we have
\begin{align}\label{S_n}
&S_n \text{ is } L_n\text{-Lipschitz function with }L_n:=\sum_{k=0}^{n-1}\alpha_k;\\ \label{R_n}
&\varphi_n(x)=R_n(x)=0 \quad \text{ and } \quad T(x)=S_n(x) \text{ for any } x\in A_n.
\end{align} 

\begin{proposition}\label{prop:Local_Lipschitzness_of_T}
	The $\boldsymbol{\alpha}$-$\boldsymbol{A}$-Takagi function is locally Lipschitz on the set $G=\bigcup_{n=0}^\infty \Int A_n$.
\end{proposition}
\begin{proof}
	Consider some point $x\in G$. Then there exists $n\in\IN_0$ such that $x\in\Int A_n$. Since $R_n(u)=0$ for any $u\in A_n$, we conclude that $T(u)=S_n(u)$ on $A_n$. Thus, $\ILip T(x)=\ILip S_n(x)\le L_n<\infty$. 
\end{proof}

\begin{proposition}\label{prop:Pointwise_Lipschitzness_of_T}
	Let $n\in \IN_0$, $x\in A_n$, and $U$ be a neighborhood of $x$ such that 
	\begin{equation*}
		d(u, A_k)\le \delta_kd(u,x)\text{ for any  } k\ge n\text{ and }u\in U.
	\end{equation*}
	Then, $T$ is pointwise Lipschitz at $x$.
\end{proposition}
\begin{proof}
It follows from \eqref{equ:Takagi2} that $L:=\sum\limits_{k=0}^\infty\alpha_k\delta_k<\infty$. Since $R_n(x)=0$, we have 
	\begin{equation*}
		\big|R_n(u)-R_n(x)\big|=R_n(u)=\sum_{k=n}^\infty\alpha_kd(u,A_k)
		\le\sum_{k=n}^\infty\alpha_k\delta_kd(u,x)\le L\,d(u,x)
	\end{equation*}
 for every $u\in U$.	Therefore, $R_n$ is pointwise Lipschitz at $x$. Since $S_n$ is Lipschitz, it follows that $T=S_n+R_n$ is also pointwise Lipschitz at $x$.
\end{proof}

\section{Non-pointwise Lipschitzness of  $T$  outside of $\bigcup\boldsymbol{A}$}

\begin{proposition}\label{prop:Lip_infinity_outside_of_A}
Let us suppose that $x\in X\setminus \bigcup\boldsymbol{A}$ and 
\begin{equation}\label{equ:Lip_infinity_outside_of_A}
	\limsup_{n\to\infty}\Big(\alpha_n-\sum_{k=0}^{n-1}\alpha_k\Big)=\infty.
\end{equation}	
Then, $\Lip T(x)=\infty$. In particular, if $a>2$ and $\alpha_n=a^n$ for every $\mathbb{N}_0$, then \eqref{equ:Lip_infinity_outside_of_A} holds.
\end{proposition}
\begin{proof}
Recall  that $L_n=\sum\limits_{k<n}\alpha_k$ for any $n\in\IN_0$. Choose  an infinite set $M\subseteq\IN_0$ such that
\begin{equation*}\label{equ:Lip_infinity_outside_of_A(no1)}
	\gamma_n:=\alpha_n-L_n\to\infty\text{ as }n\to \infty, n\in M
\end{equation*}	
and $\gamma_n>1$ for $n\in M$. Since $\alpha_n\ge \gamma_n$, we have that $\alpha_n\to \infty$ as $n\to \infty$, $n\in M$. Set $\lambda_n=1-\frac{1}{\alpha_n^2}$. For any $n\in M$ there exists $u_n\in A_n$ such that 
$$
r_n:=d(x,u_n)<\tfrac{1}{\lambda_n}d(x,A_n)=\tfrac{1}{\lambda_n}\varphi_n(x).
$$ 
Therefore,
\begin{equation}\label{eq1}
	\varphi_n(x)>\lambda_nr_n\text{ for any }n\in M.
\end{equation}
By \eqref{equ:Takagi4} we have that $r_n\le\frac{\delta_n}{\lambda_n}$. As $\lambda_n\to 1$ it follows from \eqref{equ:Takagi1} that $r_n\to 0$, and hence  $u_n\to x$ as $n\to \infty$, $n\in M$. Since $u_n\in A_n$, we obtain by \eqref{R_n} that $R_n(u_n)=0$. Therefore, using \eqref{S_n} and \eqref{eq1},  for any $n\in M$ we have
\begin{align*}
	\big|T(u_n)-T(x)\big|&\ge \big|R_n(u_n)-R_n(x)\big|-\big|S_n(u_n)-S_n(x)\big|\\
	&\ge R_n(x)-L_nr_n\ge \alpha_n\varphi_n(x)-L_nr_n\\
	&\ge \big(\lambda_n\alpha_n-L_n\big)r_n
	=\Big(\big(1-\tfrac{1}{\alpha_n^2}\big)\alpha_n-L_n\Big)r_n\\
	&=\big(\gamma_n-\tfrac{1}{\alpha_n}\big)r_n.
\end{align*}
Thus, we conclude
\begin{equation*}
	\Lip T(x)\ge\limsup\limits_{\vspace{2pt}\substack{\vspace{1pt}n\to \infty\\ \vspace{1pt} n\in M}}\frac{\big|T(u_n)-T(x)\big|}{r_n}\ge\lim\limits_{\vspace{2pt}\substack{\vspace{1pt}n\to \infty\\ \vspace{1pt} n\in M}}\big(\gamma_n-\tfrac{1}{\alpha_n}\big)=\infty.
\end{equation*}

Finally,  suppose that $a>2$ and $\alpha_n=a^n$. Then, for any $n$
\begin{equation*}
	\gamma_n=a^n-\sum_{k=0}^{n-1}a^k=a^n-\frac{a^n-1}{a-1}\ge a^n-\frac{a^n}{a-1}=a^n\,\frac{a-2}{a-1}.
\end{equation*}
Hence, $\gamma_n\to\infty$ as $n\to \infty$ and \eqref{equ:Lip_infinity_outside_of_A} follows immediately.
\end{proof}

\section{Porosity and the big Lipschitz derivative of the distance function}

We begin this section by recalling the definition of porosity (see \cite{Za}, for instance). Let $E\subseteq X$ be a set and $x\in X$. The \textit{porosity}  of $E$ at $x\in X$  is defined by
\begin{equation*}
	p(E,x)=2\limsup_{\varepsilon\to0^+}\frac{p(E,x,\varepsilon)}{\varepsilon},
\end{equation*}	
where
\begin{equation*}
	p(E,x,\varepsilon)=\sup\left\{r\ge 0\colon \mbox{ there exists } u\in X\mbox{ such that } B(u,r)\subseteq B(x,\varepsilon)\setminus E\right\}.
\end{equation*}
The set $E$ is called \textit{porous} if there is $\lambda>0$ such that $p(E,x)\ge\lambda$ for any $x\in X$. Obviously, we have that $p(E,x)\ge 2$ for any $x\notin\overline{E}$. Hence, $E$ is porous if and only if 
\begin{equation*}
	p(E):=\inf_{x\in \overline{E}}p(E,x)>0.
\end{equation*}

Let us define for any $x\in\overline{E}$
\begin{equation*}
	\widetilde{p}(E,x)=\Lip d_E(x)
		\quad \text{and}\quad \widetilde{p}(E)=\inf_{x\in \overline{E}}\tilde{p}(E,x).
\end{equation*}
For convenience, we define $\widetilde{p}(E,x)=1$ for every $x\notin\overline{E}$. Since $d_E$ is $1$-Lipschitz, we have 
\begin{equation*}
	0\le \widetilde{p}(E,x)\le 1\text{ for any }x\in X. 
\end{equation*}

\begin{proposition}
Let $E\subseteq X$ and $x\in \overline{E}$. Then, 
$$
\tfrac{1}{2}p(E,x)\le \widetilde{p}(E,x)\le p(E,x).
$$
Consequently,
$
\tfrac12p(E)\le\widetilde{p}(E)\le p(E)
$	
and $E$ is porous if and only if $\tilde{p}(E)>0$.
\end{proposition}
\begin{proof}
Let $x\in \overline{E}$. First, consider $\lambda<p(E,x)$. Hence, there exists a sequence of positive ${\varepsilon_n\to 0}$ such that $\lambda<\frac{2}{\varepsilon_n}p(E,x,\varepsilon_n)$ for any $n\in\IN$. Fix some $n\in\IN$. Since $p(E,x,\varepsilon_n)>\frac{\lambda\varepsilon_n}{2}$, there are $r_n>\frac{\lambda\varepsilon_n}{2}$ and $u_n\in X$ such that $B(u_n,r_n)\subseteq B(x,\varepsilon_n)\setminus E$. Consequently, $d(u_n,E)\ge r_n>\frac{\lambda\varepsilon_n}{2}$ and $d(u_n,x)<\varepsilon_n$. Thus,
\begin{equation*}
	\widetilde{p}(E,x)=\limsup_{u\to x}\tfrac{d(u,E)}{d(u,x)}\ge\limsup_{n\to\infty}\tfrac{d(u_n,x)}{d(u_n,x)}\ge \tfrac{\lambda}{2}.
\end{equation*}
Letting $\lambda\to p(E,x)$ we obtain that $\widetilde{p}(E,x)\ge\frac{1}{2}p(E,x)$.

Now, consider $\lambda<\widetilde{p}(E,x)$. Therefore, there is a sequence of points $u_n\ne x$ such that $r_n:=d(u_n,x)\to 0$ and $\frac{d(u_n,E)}{r_n}>\lambda$ for any $n\in\IN$.  Fix $n\in \IN$. Since $d(u_n,E)>\lambda r_n$ we have that $B(u_n,\lambda r_n)\cap E=\emptyset$. Set $\varepsilon_n=2r_n$. Recall that $\lambda<\widetilde{p}(E,x)\le 1$. Hence, $d(u_n,x)+\lambda r_n=(1+\lambda) r_n<2r_n=\varepsilon_n$, which implies that $B(u_n,\lambda r_n)\subseteq B(x,\varepsilon_n)\setminus E$. Therefore, $p(E,x,\varepsilon_n)\ge \lambda r_n=\frac12\lambda\varepsilon_n$ and we obtain 
\begin{equation*}
	p(E,x)=2\limsup_{\varepsilon\to 0^+}\frac{p(E,x,\varepsilon)}{\varepsilon}
	\ge2\limsup_{n\to\infty}\frac{p(E,x,\varepsilon_n)}{\varepsilon_n}\ge\lambda.
\end{equation*}
Finally, letting $\lambda\to \tilde{p}(E,x)$ we conclude that $p(E,x)\ge\widetilde{p}(E,x)$. The result follows immediately.
\end{proof}

\begin{proposition}\label{prop:Porosity_at_isolated_point=1}
If $x$ is an isolated point of $E\subseteq X$ and $\varepsilon>0$ is such that $B(x,2\varepsilon)\cap E=\{x\}$, then $d(u,E)=d(u,x)$ for any $u\in B(x,\varepsilon)$. In particular, 	$\tilde{p}(E,x)=1$.
\end{proposition}
\begin{proof} Set $E'=E\setminus\{x\}$. Note that $d(x,E')\ge 2\varepsilon$.  Consider $u\in B(x,\varepsilon)$.
We have $d(u,E')\ge d(x,E')-d(u,x)\ge 2\varepsilon-\varepsilon=\varepsilon>d(u,x)$. Therefore,
$d(u,E)=d(u,x)$. Thus, 
\begin{equation*}
	 \widetilde{p}(E,x)=\limsup_{u\to x}\tfrac{d(u,E)}{d(u,x)}
	= \limsup_{u\to x}\tfrac{d(u,x)}{d(u,x)}
	=1.\qedhere
\end{equation*}
\end{proof}

\section{Non-pointwise Lipschitzness of  $T$}

\begin{theorem}\label{thm:Lip_infinity}
Let $x\in X$ and 
\begin{equation}\label{equ:Lip_infinity}
	\limsup_{n\to\infty}\Big(\alpha_n\tilde{p}(A_n,x)-\sum_{k=0}^{n-1}\alpha_k\Big)=\infty.
\end{equation}
Then, $\Lip T(x)=\infty$. In particular, if $\lambda>0$, $a>1+\frac1\lambda$, $\alpha_n=a^n$ for every $n\in\IN_0$, and
$\tilde{p}(A_n,x)\ge\lambda>0 $ for infinitely many $n$, then \eqref{equ:Lip_infinity} holds. 
\end{theorem}
\begin{proof} If $x\in X\setminus \bigcup\boldsymbol{A}$ then $\widetilde{p}(A_n,x)=1$ for every $n\in\IN_0$. Therefore, \eqref{equ:Lip_infinity} implies \eqref{equ:Lip_infinity_outside_of_A}. Thus, $\Lip T(x)=\infty$ by Proposition~\ref{prop:Lip_infinity_outside_of_A}.   
	
Now, consider the case where $x\in \bigcup\boldsymbol{A}$.	
Since the sequence $\boldsymbol{A}$ is increasing, there exists $n_0\in \IN_0$ such that $x\in A_{n_0}\subseteq A_n$ for every $n\ge n_0$. Denote $\lambda_n=\tilde{p}(A_n,x)$ and observe that $0\le\lambda_n\le 1$. Recall that $L_n=\sum_{k<n}\alpha_k$ for every $n$. It follows from \eqref{equ:Lip_infinity} that there exists an infinite set $M\subseteq\IN_0\cap[n_0,\infty)$ such that
\begin{equation*}
	\gamma_n:=\alpha_n\lambda_n-L_n\to\infty\, \text{ as } n\to \infty, n\in M
\end{equation*}
and $\gamma_n>0$ for $n\in M$. We have 
\begin{equation*}
	\alpha_n\ge\alpha_n\lambda_n=\gamma_n+L_n\ge \gamma_n>0
\end{equation*}
for every $n\in M$. Hence, $\alpha_n\to\infty $ as $n\to \infty$, $n\in M$. Now, fix $n\in M$ and denote $\lambda_n'=\lambda_n-\frac{1}{\alpha_n^2}$. Since  $x\in A_n$, we have 
\begin{equation*}
	\lambda_n=\Lip d_{A_n}(x)=\limsup_{u\to x}\tfrac{d(u,A_n)}{d(u,x)}>\lambda_n'.
\end{equation*}
Therefore, there is a sequence of points $u_j\ne x$ such that $r_j:=d(u_j,x)\to 0$ and
$\frac{1}{r_j}d(u_j,A_n)>\lambda_n'$. It follows that $\varphi_n(u_j)=d(u_j,A_n)>\lambda_n'r_j$ for every $j\in\IN$. On the other hand, as $x\in A_n$, we have  $R_n(x)=0$ by \eqref{R_n}.

It follows from the above observations that
\begin{equation*}
	\left|R_n(u_j)-R_n(x)\right|=R_n(u_j)\ge \alpha_n\varphi_n(u_j)>\alpha_n\lambda_n'r_j
	=\left (\alpha_n\lambda_n-\tfrac{1}{\alpha_n}\right )r_j
\end{equation*}	
for every $j$. Therefore, \eqref{S_n} yields that 
\begin{align*}
	\left|T(u_j)-T(x)\right|&\ge \left|R_n(u_j)-R_n(x)\right|-\left|S_n(u_j)-S_n(x)\right|\\
	&\ge \left(\alpha_n\lambda_n-\tfrac{1}{\alpha_n}\right)r_j -L_nr_j
	= \left(\gamma_n-\tfrac{1}{\alpha_n}\right)r_j
\end{align*}
for every $j$. Hence,
\begin{equation*}
	\Lip T(x)\ge\limsup_{j\to\infty }\frac{\left |T(u_j)-T(x)\right|}{d(u_j,x)}
	\ge \gamma_n-\tfrac{1}{\alpha_n}
\end{equation*}
for every $n\in M$, which implies $\Lip T(x)\ge\lim\limits_{\vspace{2pt}\substack{\vspace{1pt}n\to \infty\\ \vspace{1pt} n\in M}}\big(\gamma_n-\tfrac{1}{\alpha_n}\big)=\infty$.

Finally, suppose that  $\lambda>0$, $a>1+\frac1\lambda$, $\alpha_n=a^n$ for every $n\in\IN_0$ and $\widetilde{p}(A_n,x)\ge\lambda>0 $ for infinitely many $n$. For each such $n$, we have
\begin{equation*}
	\gamma_n\ge \lambda a^n-\sum_{k=0}^{n-1}a^k=\lambda a^n-\frac{a^n-1}{a-1}\ge \lambda a^n-\frac{a^n}{a-1}=\lambda \,a^n\frac{a-\left(1+\frac{1}{\lambda}\right)}{a-1}\to \infty
\end{equation*}
and then \eqref{equ:Lip_infinity} holds.
\end{proof}

\section{Maximal separated sets with respect to a function}

Throughout this section, we fix a function $h\colon X\to [0,\infty)$. We say that a set $A\subseteq X$ is \textit{$h$-separated} if
\begin{equation}\label{equ:Definition_of_h-separated_set}
	d(x,y)\ge\min\big\{h(x),h(y)\big\}\text{ for any distinct points }x,y\in A.
\end{equation}
In the case where $h(x)=\varepsilon$ for any $x\in X$, $h$-separated sets are exactly $\varepsilon$-separated set.

A set $A\subseteq X$ is called \textit{maximal $h$-separated} if it is $h$-separated and for any $h$-separated set $A'$ in $X$ the condition $A\subseteq A'$ implies $A=A'$.

\begin{proposition}\label{prop:existing_of_maximal_h-separated_set}
	Let $B$ be an $h$-separated set in $X$. Then there exists a maximal $h$-separated set $A$ in $X$ such that $B\subseteq A$. 
\end{proposition}
\begin{proof}
Clearly, the property of being $h$-separated is of finite character: $C$ is $h$-separated if and only if every finite subset of $C$ is $h$-separated. Therefore, the result follows from the Teichmüller--Tukey lemma (see \cite{Eng}, for instance).	
\end{proof}

\begin{proposition}\label{prop:characterisation_of_maximal_h-separated_sets}
	Let $A$ be an $h$-separated set in $X$. Then  $A$ is maximal $h$-separated if and only if 
	\begin{equation}\label{equ:characterisation_of_maximal_h-separated_sets}
		X\setminus A=\big\{x\in X:d(x,y)<\min\big\{h(x),h(y)\big\}\text{ for some }y\in A\setminus\{x\}\big\}.
	\end{equation}
\end{proposition}
\begin{proof} Let $RH$ denotes the right-hand side of \eqref{equ:characterisation_of_maximal_h-separated_sets}. 
	
Suppose that $A$ is maximal $h$-separated. To prove $X\setminus A\subseteq RH$ we
consider some $x\in X\setminus A$. Suppose that  $x\notin RH$. Therefore, inequality \eqref{equ:Definition_of_h-separated_set} holds for any $y\in A$. This means that $A':=A\cup\{x\}$ is $h$-separated. Thus, the maximality of $A$ gives $A'=A$ and consequently, $x\in A$, which is impossible. Therefore, $X\setminus A\subseteq RH$. 

Conversely, consider $x\in RH$. Then there is $y\in A\setminus\{x\}$ such that $d(x,y)<\min\big \{h(x),h(y)\big \}.$ Hence, inequality \eqref{equ:Definition_of_h-separated_set} does not hold. Since $A$ is $h$-separated, we conclude that $x\in X\setminus A$. Thus, $RH\subseteq X\setminus A$.

Now, we suppose that $A$ is an $h$-separated set such that \eqref{equ:characterisation_of_maximal_h-separated_sets} holds. Let $A'$ be an $h$-separated set with $A\subseteq A'$. Assume that there exists $x\in A'\setminus A$. It follows from \eqref{equ:characterisation_of_maximal_h-separated_sets} that there is $y\in A\setminus\{x\}$ such that $d(x,y)<\min\big\{h(x),h(y)\big\}$. Therefore, inequality   \eqref{equ:Definition_of_h-separated_set} does not hold for such distinct points $x,y\in A'$. This contradics the fact that $A'$ is $h$-separated. Thus, $A'=A$ and hence, $A$ is maximal. 
\end{proof}

\begin{proposition}\label{prop:maximal_h-separated}
	Let $A$ be a maximal $h$-separated set in $X$, $F=h^{-1}(0)$, and $G=X\setminus F$. Then the following statements hold:
	\begin{enumerate}[label=$(\roman*)$]
		\item\label{prop:maximal_h-separated(item1)} $F\subseteq A$;
		\item\label{prop:maximal_h-separated(item2)} $d(x,A)\le h(x)$ for any $x\in X$.
\end{enumerate}
Moreover, in the case where $h$ is lower semicontinuous we have that
\begin{enumerate}[label=$(\roman*)$,start=3]		
		\item\label{prop:maximal_h-separated(item3)} $A$ is closed;
		\item\label{prop:maximal_h-separated(item4)} if  $x\in G\cap A$, then $x$ is an isolated point of $A$; 
		\item\label{prop:maximal_h-separated(item5)} if  $x\in F\cap\overline{G}$ and $h$ continuous at $x$, then $x\in\overline{G\cap A}$. 
	\end{enumerate}
\end{proposition}
\begin{proof}
	\ref{prop:maximal_h-separated(item1)}. Consider $x\notin A$. By Proposition \ref{prop:characterisation_of_maximal_h-separated_sets},  there exists $y\in A\setminus\{x\}$ such that $d(x,y)<\min\big \{h(x),h(y)\big \}$. Therefore, $h(x)>0$ and we have $X\setminus A\subseteq X\setminus h^{-1}(0)=X\setminus F$, which yields \ref{prop:maximal_h-separated(item1)}.
	
	\ref{prop:maximal_h-separated(item2)}. Fix $x\in X$. If $x\in A$ then $d(x,A)=0\le h(x)$. Let us suppose that $x\in X\setminus A$.  By Proposition \ref{prop:characterisation_of_maximal_h-separated_sets} again, there exists $y\in A$ such that 
	$d(x,y)<\min\big\{h(x),h(y)\big\}$. Thus, $d(x,A)\le d(x,y)<h(x)$.
	
	Now, we suppose that $h$ is lower semicontinuous.
	
	\ref{prop:maximal_h-separated(item3)}. Let us prove that $X\setminus A$ is open. Consider $x\in X\setminus A$. It follows from Proposition~\ref{prop:characterisation_of_maximal_h-separated_sets} that there exists $y\in A\setminus\{x\}$ such that 
	$d(x,y)<\min\big\{h(x),h(y)\big\}$. Define $f:X\to\IR$ by
	\begin{equation}\label{equ:maximal_h-separated(function_f)}
		f(u)=\min\big\{h(u),h(y)\big\}-d(u,y),\ \ \ u\in X.
	\end{equation}
	Clearly,  $f$ is lower semicontinuous and $f(x)>0$. Consequently, there is a neighborhood $U$ of $x$ such that $y\notin U$ and $f(u)>0$ on $U$. Hence, $d(u,y)<\min\big\{h(u),h(y)\big\}$ for any $u\in U$. Therefore, $U\subseteq X\setminus A$ by  Proposition \ref{prop:characterisation_of_maximal_h-separated_sets}. Thus, $X\setminus A$ is open.
	
	\ref{prop:maximal_h-separated(item4)}. Let $x\in A$ and $h(x)>0$. Set $\varepsilon=\frac{1}{2}h(x)$. By lower semicontinuity of $h$ there exists a neighborhood $U$ of $x$ such that $h(u)>\varepsilon$ on $U$ and $U\subseteq B(x,\varepsilon)$. Suppose that there exists $y\in U\cap A$ such that $y\ne x$. Since $A$ is $h$-separated and $y\in U\subseteq B(x,\varepsilon)$, we have $\varepsilon>d(x,y)\ge\min\big\{h(x),h(y)\big\}\ge\varepsilon$. This contradiction shows that $y=x$. Thus, $U\cap A=\{x\}$, and consequently, $x$ is an isolated point of $A$. 
	
	\ref{prop:maximal_h-separated(item5)}. Consider some point $x\in F\cap\overline{G}$ such that $h$ continuous at $x$. Therefore, $h(x)=0$ and there is a sequence of points $x_n\in G$ such that $x_n\to x$. Hence, $\lim\limits_{n\to\infty}h(x_n)=h(x)=0$. Let $M:=\big\{n\in\IN: x_n\in A\big\}$. If $M$ is infinite, then 
$$
x=\lim\limits_{\vspace{2pt}\substack{\vspace{1pt}n\to \infty\\ \vspace{1pt} n\in M}}x_n\in\overline{G\cap A}.
$$
	
Let us consider the case where $M$ is finite. Therefore, $N:=\IN\setminus M$ is infinite. Consider $n\in N$. Then $x_n\in X\setminus A$. By Proposition \ref{prop:characterisation_of_maximal_h-separated_sets} there exists $y_n\in A$ such that $d(x_n,y_n)<\min\big\{h(x_n), h(y_n)\big\}$. Therefore, $h(y_n)>0$, and hence, $y_n\in G\cap A$. On the other hand, we have $d(x_n,y_n)<h(x_n)$ and consequently, 
$$
d(x,y_n)\le d(x,x_n)+d(x_n,y_n)<d(x,x_n)+h(x_n)\to 0\text{ as }\, n\to \infty ,\, n\in N.
$$ 
Hence, we conclude
\[
x=\lim\limits_{\substack{\vspace{1pt}n\to \infty\\ \vspace{1pt}n\in N}}x_n\in\overline{G\cap A}. \qedhere
\]
\end{proof}
\section{Characterization of sets where the big Lipschitz derivative is infinite}

In this section, we construct a special sequence $\boldsymbol{A}=\boldsymbol{A}^{\boldsymbol{G}}$ that allows us to determine the sets $(\ILip T)^{-1}(\infty)$ and $(\Lip T)^{-1}(\infty)$.

Let us consider a decreasing sequence $\boldsymbol{G}=(G_n)_{n=0}^\infty$ of open sets $G_n$ in $X$ and denote 
\begin{equation}\label{equ:G_delta-pair}
	E=\bigcap_{n=0}^\infty G_n\ \ \text{ and }\ \ F=\bigcap_{n=0}^\infty \overline{G}_n.
\end{equation}
Consider $b>a>2$ and let
\begin{equation}\label{equ:alpha_and_delta}
	\alpha_n=a^n
	\ \ \ \text{and}\ \ \ 
	\delta_n=\frac{1}{b^n}
\end{equation}
for any $n\in\IN_0$. Denote $F_n=X\setminus G_n$  and define a function $h_n\colon X\to[0,\infty)$ by
\begin{equation*}
	h_n(x)=\delta_n \min\big\{1,d(x,F_n)\big\}
\end{equation*}
for any $x\in X$. It is immediate that $h_n$ is $\delta_n$-Lipschitz. 

Let us define inductively an increasing sequence of maximal $h_n$-separated sets $A_n$.  Consider $B=\{b\}$ where $b$ is an arbitraty point of $X$. By Proposition~\ref{prop:existing_of_maximal_h-separated_set} there exists a maximal $h_0$-separated set $A_0\supseteq B$. Suppose that for some $n\in \IN$ we already constructed a maximal $h_{n-1}$-separated set $A_{n-1}$. Since $\delta_n<\delta_{n-1}$ and $F_n\supseteq F_{n-1}$, we obtain that $h_n(x)\le h_{n-1}(x)$ for any $x\in X$. Therefore, $A_{n-1}$ is also $h_n$-separated. By Proposition~\ref{prop:existing_of_maximal_h-separated_set} there exists a maximal $h_n$-separated set $A_n\supseteq A_{n-1}$. Thus, the required sequence $\boldsymbol{A}^{\boldsymbol{G}}=(A_n)_{n=0}^\infty$ is defined. 

By construction we have \eqref{equ:Takagi5}. Since $b>a>2$, properties \eqref{equ:Takagi1} and \eqref{equ:Takagi2} hold. By Proposition~\ref{prop:maximal_h-separated}\ref{prop:maximal_h-separated(item3)}, property \eqref{equ:Takagi3} also holds. It follows from
Proposition~\ref{prop:maximal_h-separated}\ref{prop:maximal_h-separated(item2)} that $d(x,A_n)\le h_n(x)\le\delta_n$ for any $x\in X$ and $n\in\IN_0$. Therefore, \eqref{equ:Takagi4} holds. Thus, all the basic assumptions \eqref{equ:Takagi1}-\eqref{equ:Takagi4} are satisfied.

\begin{theorem}\label{thm:ILip_and_Lip_of_Takagi}
	Let $\boldsymbol{\alpha}$ and $\boldsymbol{A}=\boldsymbol{A}^{\boldsymbol{G}}$ be the sequences constructed above, $E$ and $F$ be the sets defined in \eqref{equ:G_delta-pair}, and $T=T_{\boldsymbol{\alpha},\boldsymbol{A}}$ be the $\boldsymbol{\alpha}$-$\boldsymbol{A}$-Takagi function. Then, we have  
$$
E=(\Lip T)^{-1}(\infty) \quad \text{ and } \quad \overline{E}\subseteq(\ILip T)^{-1}(\infty)\subseteq F.
$$
\end{theorem}
\begin{proof}
It follows from Proposition~\ref{prop:Local_Lipschitzness_of_T} that $T$ is locally Lipschitz on $G=\bigcup_{n=0}^{\infty}\Int A_n$. 
Observe that $h_n^{-1}(0)=F_n$. Thus,  Proposition~\ref{prop:maximal_h-separated}\ref{prop:maximal_h-separated(item1)} 
implies that  $F_n\subseteq A_n$. Therefore, 
\begin{equation*}
	X\setminus F=X\setminus\bigcap_{n=0}^\infty \overline{G}_n
	=\bigcup_{n=0}^\infty \big(X\setminus\overline{G}_n\big)
	=\bigcup_{n=0}^\infty\Int F_n
	\subseteq \bigcup_{n=0}^\infty\Int A_n=G
\end{equation*}
and consequently, $T$ is locally Lipschitz on $X\setminus F$, that is $\ILip T(x)<\infty$ on  $X\setminus F$. Hence, $(\ILip T)^{-1}(\infty)\subseteq F$.

Now, consider some $x\in X\setminus E=\bigcup_{n=0}^\infty F_n$ and choose $n\in \IN_0$ such that $x\in F_n$. We have $x\in F_n\subseteq F_k$ for any $k\ge n$. It follows from Proposition \ref{prop:maximal_h-separated}~\ref{prop:maximal_h-separated(item2)} that for any $u\in X$ and $k\ge n$ we have
\begin{equation*}
	d(u,A_k)\le h_k(u)\le \delta_k d(u,F_k)\le\delta_k d(u,x). 
\end{equation*}
Thus, Proposition~\ref{prop:Pointwise_Lipschitzness_of_T} implies that $\Lip T(x)<\infty$.

Proposition~\ref{prop:maximal_h-separated}\ref{prop:maximal_h-separated(item4)}  guarantees that $A_n\setminus F_n$ is discrete. Therefore, Proposition~\ref{prop:Porosity_at_isolated_point=1} implies that $\tilde{p}(A_n,x)=1$ for any $x\in A_n\setminus F_n$. By definition, $\tilde{p}(A_n,x)=1$ on $X\setminus A_n$. Thus, $\tilde{p}(A_n,x)=1$ for any $x\in G_n=X\setminus F_n$. This implies that $\tilde{p}(A_n,x)=1$ for any $x\in E$ and $n\in\IN_0$. Hence, we can apply Theorem~\ref{thm:Lip_infinity} with $\lambda=1$ to conclude that $\Lip T(x)=\infty$ for any $x\in E$. Since $\ILip T\ge \Lip T$ and $\ILip T$ is upper semicontinuous, we conclude that $\ILip T(x)=\infty$ on $\overline{E}$. Thus, $\overline{E}\subseteq(\ILip T)^{-1}(\infty)$.
\end{proof}

We are now in a position to prove the main result of the paper.

\begin{theorem}\label{thm:L-problem}
	Let $X$ be a metric space without isolated points and $E\subseteq X$. Then the following conditions are equivalent:
	\begin{enumerate}[label=$(\roman*)$]
		\item\label{thm:L-problem(item1)}
		$E$ is a $G_\delta$-set in $X$;
		\item\label{thm:L-problem(item2)}
		there exists a continuous function $f\colon X\to\IR$ such that $(\Lip f)^{-1}(\infty)=E$;
		\item\label{thm:L-problem(item3)} 
		there exists a continuous function $f\colon X\to\IR$ such that $(\Lip f)^{-1}(\infty)=E$ and $(\ILip f)^{-1}(\infty)=\overline{E}$.
	\end{enumerate}
\end{theorem}
\begin{proof}
\ref{thm:L-problem(item3)}$\Rightarrow$\ref{thm:L-problem(item2)}. This implication is evident.

\ref{thm:L-problem(item2)}$\Rightarrow$\ref{thm:L-problem(item1)}. It follows from Corollary 5.6  of \cite{MaWo2026-1}  that $(\Lip f)^{-1}(\infty)$ is of type $G_\delta$. 

\ref{thm:L-problem(item1)}$\Rightarrow$\ref{thm:L-problem(item3)}. Since $E$ is a $G_\delta$-set, there is a sequence of open sets $G_n'$ such that $E=\bigcap_{n=0}^\infty G_n'$. Define the sequence $\boldsymbol{G}=(G_n)_{n=0}^\infty$ by
\begin{equation*}
	G_n=\Big\{x\in\bigcap_{k\le n}G_k'\colon d(x, E)<\tfrac1n\Big\},\ \ \ n\in \IN_0.
\end{equation*}
Clearly, $\boldsymbol{G}$ is decreasing and
$$
\bigcap_{n=1}^{\infty}\overline{G}_n\subseteq \bigcap_{n=1}^{\infty}\big \{x\in X :d(x, E)\leq \tfrac1n\big \}=\overline{E}.
$$
 Thus, \eqref{equ:G_delta-pair} holds with $F=\overline{E}$. 
Let $b>a>2$, $\boldsymbol{\alpha}$  and $\boldsymbol{\delta}$ by defined by \eqref{equ:alpha_and_delta}, $\boldsymbol{A}=\boldsymbol{A}^{\boldsymbol{G}}$, and $f=T_{\boldsymbol{\alpha},\boldsymbol{A}}$ be the $\boldsymbol{\alpha}$-$\boldsymbol{A}$-Takagi function. By Theorem~\ref{thm:ILip_and_Lip_of_Takagi} we obtain that $E=(\Lip f)^{-1}(\infty)$ and $F=\overline{E}=(\ILip f)^{-1}(\infty)$.
\end{proof}

We conclude the paper with the following open question:
\begin{problem}
Let $X$ be a metric space without isolated points. Characterize the pairs $(F,E)$, where $F$ is closed in $X$ and $E\subseteq F$ is a $G_\delta$-set in $X$, for which there exists a continuous function $f\colon X\to\IR$ such that $(\ILip f)^{-1}(\infty)=F$ and $(\Lip f)^{-1}(\infty)=E$.
\end{problem}

\section{Acknowledgments}
This paper was partially written while the authors were participating
in the VII and VIII Workshops on Postmodern Real Analysis in
Bedlewo, Poland. The authors would like to thank the organizers
of both workshops for the invitation and for funding their participation.

\bibliographystyle{abbrv}
\bibliography{RefsLLMW}
\end{document}